\documentclass[10pt]{article}
\usepackage{latexsym}
\usepackage{graphicx}
\usepackage{color,colortab}
\usepackage{amsmath}
\usepackage{amsthm}
\usepackage{amsfonts}
\usepackage{amssymb}
\usepackage{enumerate}

\definecolor{darkblue}{RGB}{0, 0, 102}
\definecolor{IKB}{RGB}{0, 47, 167}

\newtheorem{theorem}{Theorem}
\newtheorem*{theorem*}{Theorem}
\newtheorem{lemma}[theorem]{Lemma}

\newtheorem{cor}[theorem]{Corollary}

\newtheorem{claim}[theorem]{Claim}
\newtheorem{remark}[theorem]{Remark}

\newcommand{\N}{\ensuremath{\mathbb{N}}}

\newcommand{\Z}{\ensuremath{\mathbb{Z}}}
\newcommand{\E}{\begin{equation}}
\newcommand{\EE}{\end{equation}}

\newcommand{\R}{\mathbb{R}}

\newcommand{\proj}{\operatorname{proj}}

\newcommand{\conv}{\operatorname{conv}}
\newcommand{\relint}{\operatorname{relint}}

\newcommand{\intr}{\operatorname{int}}
\newcommand{\aff}{\operatorname{aff}}
\newcommand{\cone}{\operatorname{cone}}

\newenvironment{cpf}{\begin{trivlist} \item[] {\em Proof of claim.}}{\hspace*{\stretch{1}} $\diamond$ \end{trivlist}}

\advance \topmargin by -\headheight
\advance \topmargin by -\headsep
\begin{document}

\title{Balas formulation for the union of polytopes is optimal}

\author{Michele Conforti\thanks{Dipartimento di Matematica ``Tullio Levi-Civita'', Universit\`a degli studi di Padova, Italy. Supported by the grant SID 2016.}, Marco Di Summa\footnotemark[1], Yuri Faenza\thanks{IEOR, Columbia University. Supported by a SEAS-IEOR starting grant and by the Swiss National Science Foundation.}}

\maketitle

\begin{abstract} A celebrated theorem of Balas gives a linear mixed-integer formulation for the union of two nonempty polytopes whose  relaxation gives the convex hull of this union. The number of inequalities in Balas formulation is linear in the number of inequalities that describe the two polytopes and the number of variables is doubled.

In this paper we show that this is best possible: in every dimension there exist two nonempty polytopes such that if a formulation for the convex hull of their union has a number of inequalities that is polynomial in the  number of inequalities that describe the two polytopes, then the number of additional variables is at least linear in the dimension of the polytopes.

% in any formulation for the convex hull of the union of two nonempty polytopes  whose  number of inequalities is polynomial in the  number of inequalities that describe the two polytopes,  the number of additional variables is at least linear in the dimension of the polytopes.

We then show that this result essentially carries over if one wants to approximate the convex hull of the union of two polytopes and also in the more restrictive setting of lift-and-project.

\end{abstract}

\section{Introduction}

Linear extensions are a  powerful tool in  linear optimization, since they allow to reduce an optimization problem over a polyhedron $P$ to an analogous problem over a second polyhedron $Q$, that may be describable with a smaller system of linear constraints. For this reason, a number of recent studies (e.g.~\cite{Steurer,Sam,KaibelKostya,Thomas}) focus on proving upper and lower bounds on the \emph{extension complexity} of a polytope $P$, i.e., on the minimum number of linear inequalities needed to describe a linear extension of $P$. Bounds on the extension complexity hence guarantee (or disprove) the theoretical efficiency of linear programming methods for certain optimization problems.

For practical purposes, the number of additional variables used in a linear extension is also an important parameter, see e.g.~\cite{WolseyVV}.  This paper studies   the minimum number of variables needed to obtain a linear extension for the convex hull of the union of two  polytopes, where the number of inequalities describing the linear extension is polynomially bounded with respect to the number of inequalities in the descriptions of the two polytopes.

\subsection{Balas formulation for the union of polytopes}\label{sec:Balas}

For any set $X\subseteq\R^d$, we denote by $\conv(X)$ the convex hull of $X$.
We recall the following theorem of Balas~\cite{Balas}.

\begin{theorem}\label{thr:balas}  Let
 $P_1:=\{x \in \R^d: A_1x \leq b_1\}$ and $P_2:=\{x \in \R^d : A_2 x \leq b_2\}$ be nonempty polytopes. Then
\begin{equation}\label{eq-balas}
\begin{split}
\conv(P_1 \cup P_2)=\{ & x\in\R^d : \exists\, (x_1,x_2,\lambda)\in\R^d\times\R^d\times\R \mbox{ s.t. } \\
& x=x_1+x_2;\,Ax_1 \leq \lambda b_1; \, Ax_2 \leq (1-\lambda) b_2; \, 0\leq\lambda\leq 1\}.
\end{split}
\end{equation}
\end{theorem}

The above result can be seen as follows. By definition of the convex hull operator, we have that
\begin{equation}\label{eq-balas2}
\conv(P_1 \cup P_2)=\{x: \exists\, (y_1,y_2,\lambda) \mbox{ s.t. } x=\lambda y_1+(1-\lambda) y_2;\,Ay_1 \leq  b_1; \, Ay_2 \leq  b_2; \, 0\leq\lambda\leq 1\}.
\end{equation}
Since $P_1$ and $P_2$ are nonempty polytopes, $\{x \in \R^d: A_1x \leq 0\}=\{x \in \R^d : A_2 x \leq 0\}=\{0\}$.
Therefore it is easy to argue that  the system in \eqref{eq-balas2} can be linearized by substituting $x_1:=\lambda y_1$ and $x_2:=(1-\lambda)y_2$ to obtain the system in \eqref{eq-balas}.  We refer to~\cite[Theorem 4.39]{CCZ} for the case in which $P_1$ and $P_2$ are (possibly empty) polyhedra. Note that as $\conv(P_1\cup \dots\cup P_k)=\conv(\conv (P_1\cup \dots \cup P_{k-1})\cup P_k)$, restricting to the case $k=2$ is with no loss of generality.\medskip

Theorem \ref{thr:balas} is fundamental for the geometric approach to Integer Programming, where, given a polytope $P\subseteq \R^d$,  tighter and tighter polyhedral relaxations of the set $S:=P\cap \Z^d$ are obtained as convex hulls of union of polytopes.
An example of this paradigm is as follows. Given  $(\pi,\pi_0)\in \Z^d\times\Z$, let
\begin{equation}\label{eq:split}
P_0 :=\{x \in P : \pi x \leq \pi_0\} \; \hbox{ and } \; P_1 :=\{x \in P : \pi x \geq \pi_0+1\}.\end{equation}
Then $\conv(P_0\cup P_1)\cap \Z^d=S$ and $\conv(P_0\cup P_1)\subseteq P$. This second containment is strict if and only if $\pi_0<\pi v< \pi_0+1$  for some vertex $v$ of $P$.  In this case  $\conv(P_0\cup P_1)$ is a tighter polyhedral relaxation for $S$. The \emph{split cuts} used in Integer Programming, see e.g. Chapter 5 in \cite{CCZ}, are the linear inequalities that are valid for $\conv(P_0\cup P_1)$, for some  $(\pi,\pi_0)\in \Z^d\times\Z$.
\medskip

Given  polytope $P\subseteq \R^d$,  polytope $Q\subseteq \R^{d+m}$ is a \emph{linear extension} (\emph{with $m$ additional variables}) of $P$ if there exists an affine map $\psi: \R^{d+m}\rightarrow \R^d$ such that $P=\psi(Q)$. We allow  $Q=P$. In this paper, the only affine maps we consider are orthogonal projections: i.e., $Q$ is a linear extension of $P$ when $P=\proj_x(Q):=\{x\in \R^d : \exists\, y\in \R^m \mbox{ s.t. } (x,y)\in Q\}$.  Note that restricting to orthogonal projections  is with no loss of generality.
\medskip

A system of inequalities describing a linear extension $Q$ of $P$ is a \emph{formulation} of $P$ whose \emph{size} is the number of inequalities.
The \emph{extension complexity} of $P$ is the minimum size of a formulation of $P$. Therefore, if we let
\begin{equation}\label{eq-balas3}
  Q:=\{(x,x_1,\lambda)\in \R^d\times\R^d\times\R:  Ax_1 \leq \lambda b_1; \, A(x-x_1) \leq (1-\lambda) b_2; \, 0\leq\lambda\leq 1\},
\end{equation}
then Theorem \ref{thr:balas} says that $\proj_x(Q)=\conv(P_1 \cup P_2)$. Furthermore,  given formulations of size $f_1$ and $f_2$ of nonempty polytopes $P_1$ and $P_2$,   we see that the  formulation of $Q$ given in \eqref{eq-balas3} has size $f_1+f_2+2$  and has $2d+1$ variables.  Hence the number of constraints is  linear in $f_1+f_2$ and the number of \emph{additional} variables is $d+1$. (Weltge~\cite[proof of Proposition 3.1.1]{Weltge} observed that the inequalities $0\leq \lambda \leq 1$ can be omitted from \eqref{eq-balas3} if both $P_1$ and $P_2$ have dimension at least $1$.)

The fact that the formulation of $Q$ has size $f_1+f_2+2$  has been exploited by several authors to construct small size linear extensions of polytopes that can be seen as the convex hull of the union of a polynomial number of polytopes with few inequalities. These results are  surveyed e.g. in  \cite{Conforti}, \cite{Kaibel}.
\medskip

 While most of the literature focuses on the smallest number of inequalities defining a linear extension of a given polytope, in this paper we focus on the minimum number of additional variables needed in a linear extension. More specifically, we address the following question:
\medskip

\emph{ Given  formulations of nonempty polytopes $P_1, P_2\subseteq \R^d$ with sizes $f_1$, $f_2$ respectively, let $Q$ be a linear extension of $\conv(P_1\cup P_2)$ whose  formulation has  $\mbox{poly}(f_1+f_2)$ inequalities. What is the minimum number of additional variables that $Q$ must have?}

\medskip

  We stress that  in the above question, the property $\proj_x(Q)=\conv(P_1 \cup P_2)$ must be satisfied for every choice of nonempty polytopes $P_1$, $P_2$.  If $P_1$, $P_2$ have specific properties, then few additional variables may suffice. For instance, Kaibel and Pashkovich~\cite{KaibelKostya} show that when $P_2$ is the reflection of $P_1$
  with respect to a hyperplane that leaves $P_1$ on one side, then $\conv(P_1 \cup P_2)$  admits a linear extension with only $f_1+2$ inequalities and one additional variable.

\subsection{Our contribution}

Our main result shows that if one constructs a formulation of $\conv(P_1 \cup P_2)$ whose size is polynomially bounded in the sizes of the descriptions of $P_1$ and $P_2$, then $\Omega(d)$ additional variables are needed. In other words, the construction of Balas is optimal in this respect. More specifically, we have the following:

\begin{theorem}\label{thr:main}
	Fix a polynomial $\sigma$. For each odd $d \in \N$, there exist formulations of nonempty polytopes  $P_1, P_2\subseteq \R^d $ of size $f_1$ and $f_2$ respectively, such that  any formulation of $\conv(P_1\cup P_2)$ of size at most  $\sigma(f_1+f_2)$ has $\Omega(d)$ additional variables.
\end{theorem}

We then turn to polytopes whose orthogonal projection gives an outer approximation of $\conv(P_1\cup P_2)$. Given $\varepsilon\ge0$, we say that a polytope $P'\subseteq \R^d$ is an \emph{$\varepsilon$-approximation} of a  nonempty polytope $P\subseteq \R^d$ if $P\subseteq P'$ and for all $c \in \R^d$ we have
\begin{equation}\label{eq:approx}
\underset{x \in P'}{\max}\, cx - \underset{x \in P'}{\min}\, cx \leq (1+\varepsilon) \big(\underset{x \in P}{\max}\, cx - \underset{x \in P}{\min}\, cx \big).
\end{equation}
In particular, $P$ and $P'$ have that same affine hull.   So the only  $\varepsilon$-approximation of a point is the point itself.
\medskip

 Our second result can be seen as an $\varepsilon$-approximate version of Theorem \ref{thr:main}.

\begin{theorem}\label{thr:main-approx}
Fix $\varepsilon>0$ and a polynomial $\sigma$. For each $d \in \N$, there exist formulations of nonempty polytopes  $P_1, P_2\subseteq \R^d $ of size $f_1$ and $f_2$ respectively, such that any $\varepsilon$-approximation of $\conv(P_1\cup P_2)$ of size at most  $\sigma(f_1+f_2)$ has $\Omega(d/\log d)$ additional variables.
 \end{theorem}

\subsection{MIP representations and the convex hull property}

A set $S\subseteq \R^d$ is \emph{MIP (mixed-integer programming) representable} if there exist  matrices $A$, $B$, $C$ and a vector $d$ such that $S=\proj_x(Q)$, where $Q:=\{(x,y,z):Ax+By+Cz\le d,\,z\mbox{ integral}\}$.  Under the condition that matrices $A$, $B$, $C$ and  vector $d$ be rational, the MIP representable sets were characterized by Jeroslow and Lowe~\cite{JeroslowLowe}, see also Basu et al.~\cite{Basuetal} for a different characterization. We refer to the recent survey of  Vielma \cite{Vielma} on MIP representability.

If $P_1$ and $P_2$ are nonempty polytopes, then $P_1\cup P_2$ is a MIP representable set. Indeed a MIP representation of this set can be obtained by imposing integrality on variable $\lambda$ in the system in \eqref{eq-balas3}; see \cite{Balas}.
This is not the only representation of $P_1\cup P_2$: the famous big-$M$ method gives a representation with $f_1+f_2+2$ inequalities and only $d+1$ variables, where $f_1$ and $f_2$ are the sizes of formulations of $P_1$, $P_2$. So this representation is more compact than the one given by Balas.

If $Q:=\{(x,y,z):Ax+By+Cz\le d,\,z\mbox{ integral}\}$ is a MIP representation of $P_1\cup P_2$, then
\begin{equation}\label{eq-convhull}
\conv(P_1\cup P_2)\subseteq {\textstyle\proj_x}(\{(x,y,z):Ax+By+Cz\le d\}).
\end{equation}
We say that a MIP representation of $P_1\cup P_2$ has the \emph{convex hull property} if the two sets in \eqref{eq-convhull} coincide. It follows from Theorem \ref{thr:balas} that the MIP representation obtained by imposing integrality on $\lambda$ in \eqref{eq-balas3} has the convex hull property and it is immediate to check that the one given by the big-$M$ method does not.

The following is a consequence of Theorem \ref{thr:main}.

\begin{theorem}\label{thr-convhull} Fix a polynomial $\sigma$. For each odd $d \in \N$, there exist formulations of nonempty polytopes  $P_1, P_2\subseteq \R^d $ of size $f_1$ and $f_2$ respectively, such that any  MIP representation of $P_1 \cup P_2$ with at most $\sigma(f_1+ f_2)$ inequalities that has the convex hull property has $\Omega(d)$ additional variables.
\end{theorem}

This  paper is organized as follows. In Section \ref{sec:strategy} we describe the main idea of our approach, which relies on a counting argument and on the existence of a polytope $P\subseteq\R^d$ that has $d^{\Omega(d)}$ facets and is the convex hull of two polytopes with polynomially (in $d$) many facets. In Section \ref{sec:tools} we develop some geometric tools for the construction of $P$, which is then obtained in Section \ref{sec:construction} via a construction using the Cayley embedding. In Section \ref{sec:Proof-of-approx} we prove Theorem \ref{thr:main-approx}, while in Section \ref{sec:split} we investigate some implications of Theorem \ref{thr:main} for the technique of lift-and-project (see Theorem \ref{thr:our-split}).

\section{An outline of the proof}\label{sec:strategy}

We assume familiarity with polyhedral theory (see e.g.~\cite{CCZ,Ziegler}). Given $S\subseteq \R^d$, we denote with $\conv(S)$, $\aff(S)$ and $\cone(S)$ its convex hull, affine hull and conical hull. We also let
$\intr(S)$, $\relint(S)$, $\dim(S):=\dim(\aff(S))$ denote the interior, relative interior, and affine dimension of $S$. A \emph{$d$-polytope} is a polytope of dimension $d$, and a \emph{$k$-face} of a polytope $P$ is a face of $P$ of dimension $k$.

Our approach to proving Theorem \ref{thr:main} is based on the following lemma relating the number of facets of a linear extension of a given polytope to the number of additional variables.  Given a $P\subseteq \R^d$, we let $f(P)$ denote the number of facets of $P$.

\begin{lemma}\label{lem:hammer}
Let $Q\subseteq \R^{d+m}$ be a $(d+m)$-polytope which is a linear extension  of a  $d$-polytope $P\subseteq \R^d$. Then $m=\Omega\left(\frac{\log f(P)}{\log f(Q)}\right)$.
\end{lemma}

\begin{proof}
For $0\le k\le d+m-1$, every $k$-face of $Q$ is the intersection of $d+m-k$ facets of $Q$ (this choice may not be unique). Then, by the binomial theorem,  the number of proper faces of $Q$ of dimension at least $d-1$ is at most
$$\sum_{j=1}^{m+1}{f(Q)\choose j}\le \sum_{j=1}^{m+1}(f(Q))^j\cdot1^{m+1-j}\le (f(Q)+1)^{m+1}.$$

Let $F$ be a facet of $P$. Then $F$ is the projection of a face $F_Q$ of $Q$ and  $\dim(F_Q)\ge \dim(F)=d-1$. Therefore the number of facets of $P$ is bounded by the number of proper faces of $Q$ of dimension at least $d-1$ and by the above argument we have that $m=\Omega\left(\frac{\log f(P)}{\log f(Q)}\right)$.
\end{proof}

We will show later (Theorem \ref{th-construction}) that for every odd $d\ge3$ there exists a $d$-polytope $P$ having  $d^{\Omega(d)}$ facets which is the convex hull of two polytopes $P_1$ and $P_2$, each having $(d-1)^2$ facets.
Let $Q$ be any linear extension of $P$ with $f(Q)=\sigma(f(P_1)+f(P_2))$.  It is well-known that $Q$ can be assumed to be full-dimensional. Since $f(P_1)=f(P_2)=(d-1)^2$, we have that $f(Q)$ is bounded by a polynomial in $d$. Let  $d+m$ be the dimension of $Q$. Then by Lemma \ref{lem:hammer}, the number of additional variables is  $m=\Omega\left(\frac{\log f(P)}{\log f(Q)}\right)=\Omega\left(\frac{d \log d}{\log d}\right)=\Omega(d)$.
This proves Theorem \ref{thr:main}.  Similar counting arguments (with different polytopes) settle Theorem \ref{thr:main-approx} and Theorem \ref{thr:our-split}.

The next two sections are devoting to proving Theorem \ref{th-construction}, which is the  missing ingredient to complete proof of Theorem \ref{thr:main}.

\section{Some tools}\label{sec:tools}

\subsection{Optimality Cones}

 We let $\mathcal{F}(P)$ be the set of the nonempty faces of a polytope $P$.

Let $P\subseteq \R^d$ be a polytope and let $F\in \mathcal{F}(P)$ be a nonempty face of $P$. An inequality $cx\le \delta$ \emph{defines} $F$ if
$$P\subseteq \{x\in \R^d:cx\le \delta\}\;\mbox{ and  }\;F=P\cap \{x\in \R^d:cx=\delta\}.$$
The \emph{optimality cone} of $F$ is the set
\[C_P(F):=\{c\in \R^d: \exists\, \delta \mbox{ s.t. }cx\le \delta \mbox{ defines }F\}.\]
We also consider the set
\[\overline{C}_P(F):=\bigcup_{F'\in \mathcal{F}(P), F'\supseteq F}C_{P}(F').\]

\begin{remark}\label{re-optcones} Let $P\subseteq \R^d$ be a  nonempty polytope. The following hold:
\begin{enumerate}
  \item $\bigcup_{F\in \mathcal{F}(P)}C_P(F)=\R^d$ and $C_{P}(F_i)\cap C_{P}(F_j)=\emptyset$ for every $F_i,F_j \in \mathcal{F}(P)$ with $F_i\ne F_j$.
  \item $C_P(P)$ is the subspace $\{c\in \R^d: \exists\, \delta \mbox{ s.t. }cx= \delta \: \forall x\in P\}.$ So $\dim(P)=d$ if and only if $C_P(P)=\{0\}$.
   \item For every $F\in \mathcal{F}(P)$, $\overline{C}_P(F)$ is the polyhedral cone generated by
	$$\{c\in \R^d: \exists\, \delta \mbox{ s.t.  } cx\le \delta \mbox{ defines  $P$ or a facet containing } F \}.$$
       \item For every $F\in \mathcal{F}(P)$, we have $C_P(F)=\relint\left(\overline{C}_P(F)\right)$.
  \item For every $F\in \mathcal{F}(P)$, $\dim (F)+\dim(C_P(F))=d$.
\end{enumerate}
\end{remark}

\begin{lemma}\label{le-subspace}  Let $P$ be a $d$-polytope. For every dimension $k$, $0\le k\le d$,
  there exists a linear subspace $V\subseteq \R^d$ such that $\dim(V)=k$ and
  $V\cap  \aff\left(\overline{C}_P(F)\right)=\{0\}$ for every $k$-face $F$ of $P$.
	\end{lemma}
	
\begin{proof}
Define
$$\mathcal{A}:=\bigcup_{\textrm{$F$ is a $k$-face of $P$} }\aff\left(\overline{C}_P(F)\right).$$
We iteratively construct subspaces $\{0\}=:V_0\subset \dots \subset V_k$ such that $\dim(V_i)=i$ and $V_i\cap  \mathcal{A}=\{0\}$, $i=0,\dots,k$.

Assume $\dim(V_i)=i$ and $V_i\cap  \mathcal{A}=\{0\}$ for some $i<k$.
For every $k$-face $F$ of $P$ we have that $\aff\left(\overline{C}_P(F)\cup V_i\right)$ is a linear space of dimension $d-k+i<d$.
Since the number of $k$-faces of $P$ is finite, the set
$$\mathcal{S}:=\bigcup_{\textrm{$F$ is a $k$-face of $P$} }\aff\left(\overline{C}_P(F)\cup V_i\right)$$
 has Lebesgue measure 0. Therefore $\R^d\setminus \mathcal{S}$ contains  a nonzero vector $v$. Let $V_{i+1}$ be the linear space generated by $V_i\cup \{v\}$. Then $\dim(V_{i+1})=i+1$ and $V_{i+1}\cap  \mathcal{A}=\{0\}$.
\end{proof}

\subsection{The polar of a cyclic polytope}

The \emph{moment curve} in $\R^d$ is defined as
$$t\mapsto x(t):=\left(\begin{array}{c}t^1\\ t^2 \\  \vdots \\ t^{d}  \end{array}\right)\in \R^d.$$
 Given pairwise distinct real numbers $t_1,\dots,t_k$,  the \emph{cyclic polytope} $P^{Cy}(d,t_1,\dots t_k)$ is $\conv(x(t_1),\dots x(t_k))$. It is well-known that, for $d$ and $k$ fixed, the combinatorial structure of a cyclic polytope does not depend on the  choice of $t_1,\dots, t_k$. So we denote such a polytope by $P^{Cy}(d,k)$. In particular (see~\cite[Section 4.7]{Grunbaum}):

\begin{lemma}\label{le-cycpoly} For $k\ge d+1$, $P^{Cy}(d,k)$ is a $d$-polytope  with $k$ vertices which is simplicial (i.e., all of its proper faces are simplices). For every subset $S$ of vertices with $|S|\le \frac{d}{2}$, $\conv(S)$ is a $(|S|-1)$-face. So for $h\le \frac{d}{2}-1$, $P^{Cy}(d,k)$ has $\binom{k}{h+1}$ $h$-faces.
  \end{lemma}

 Given   $P^{Cy}(d,k)$, $k\ge d+1$, apply a translation so that  $0\in \intr\left(P^{Cy}(d,k)\right)$. Let $D^{Cy}(d,k)$ be the polar of this translated polytope. Then by the above lemma and~\cite[Corollary 2.14]{Ziegler} we obtain:

\begin{remark}\label{re-polarcyclicpoly}  For  $k\ge d+1$, $0\in \intr\left(D^{Cy}(d,k)\right)$ and  $D^{Cy}(d,k)$ is a $d$-polytope with $k$ facets that is simple (i.e., every $h$-face, with $0\le h\le d-1$, is the intersection of exactly $d-h$ facets). For every subset $S$ of facets with $|S|\le \frac{d}{2}$, the intersection of the facets in $S$ is a $(d-|S|)$-face of $P$. So for $h\le \frac{d}{2}$, $P$ has $\binom{k}{h}$ $(d-h)$-faces.
\end{remark}

\subsection{A perturbation of the polar of a cyclic polytope}

Since for $k\ge d+1$, $0\in \intr\left(D^{Cy}(d,k)\right)$,  every valid inequality for  $D^{Cy}(d,k)$ can be written in the form $ax \leq 1$. Assume now $d$ even and $k=d^2$. Hence $D^{Cy}(d,d^2)$ has $d^2$ facets, and we  arbitrarily partition the normals to its facets into $d/2$ \emph{color classes} of size $2d$, so that $D^{Cy}(d,d^2)$ can be written as
\[\{x\in \R^d: a^i_jx\le 1, \,i=1,\dots, 2d,\, j=1,\dots, d/2\},\]
where for every $j$ the vectors $a_j^1,\dots,a_j^{2d}$ are the normals that belong to color class $j$.

By Lemma \ref{le-subspace}, there exists a linear subspace $V\subseteq \R^d$ such that $\dim(V)=d/2$ and
$$V\cap \aff\big(\overline{C}_{D^{C_y}(d,d^2)}(F)\big)=\{0\}$$
for each $(d/2)$-face $F$ of $D^{C_y}(d,d^2)$. Let  $u_1,\dots, u_{d/2}\in V$ be such that $\conv(u_1,\dots, u_{d/2})$ is a $(d/2-1)$-simplex and $0\in \relint(\conv(u_1,\dots, u_{d/2}))$. Note that the norms of vectors $u_i$ can be arbitrarily small.  Consider the following polytope:
\[Q^{Cy}(d,d^2):=\{x\in \R^d: (a^i_j+u_j )x\le 1, \,i=1\dots, 2d,\, j=1,\dots,d/2\}.\]

\begin{remark}\label{rem:perturb}
  Since, by Remark \ref{re-polarcyclicpoly}, $D^{Cy}(d,d^2)$ is a simple $d$-polytope,  $u_1,\dots, u_{d/2}$ can be scaled so that the combinatorial structures of $D^{Cy}(d,d^2)$ and $Q^{Cy}(d,d^2)$ coincide (see Section 2.5 in~\cite{Ziegler}). So the properties of Remark \ref{re-polarcyclicpoly} hold for $Q^{Cy}(d,d^2)$ as well, and there is an isomorphism between the face lattices of $D^{Cy}(d,d^2)$ and $Q^{Cy}(d,d^2)$.
\end{remark}

Call a $(d/2)$-face of $D^{Cy}(d,d^2)$ \emph{colorful} if it is the intersection of $d/2$ facets, no two of them having the same color. More precisely, a face $F$ is colorful if there exist indices $i_j,\, j=1,\dots,d/2$, such that:
\begin{equation}\label{eq-colorfulFD}
F=\big\{x\in D^{Cy}(d,d^2):a^{i_j }_jx=1,\,j=1,\dots,d/2\big\}.
\end{equation}

Given a colorful face $F$ described as above, let
\begin{equation}\label{eq-colorfulFQ}
F':=\big\{x\in Q^{Cy}(d,d^2):(a^{i_j }_j+u_j)x=1,\,j=1,\dots, d/2\big\}
\end{equation}
be the corresponding colorful face of $Q^{Cy}(d,d^2)$. Because of Remark \ref{rem:perturb}, $F'$ has dimension $d/2$.

\begin{lemma}\label{le-intoptcones}
Let $F$ and $F'$ be corresponding colorful  faces of $D^{Cy}(d,d^2)$ and  $Q^{Cy}(d,d^2)$ respectively. Then
\[C_{D^{Cy}(d,d^2)}(F)\cap C_{Q^{Cy}(d,d^2)}(F')=\{\lambda r,\, \lambda > 0\}\]
for some $r\in \R^d\setminus \{0\}$.
\end{lemma}

\begin{proof}
Assume that $F$ is given as in \eqref{eq-colorfulFD}. Since $D^{Cy}(d,d^2)$ and $Q^{Cy}(d,d^2)$ are $d$-polytopes by Remarks \ref{re-polarcyclicpoly} and \ref{rem:perturb}, by Remark \ref{re-optcones} we have that
\begin{equation}\label{eq-conerelations1}
\overline{C}_{D^{Cy}(d,d^2)}(F)=\cone\big(a^{i_j }_j,\,j=1,\dots, d/2\big),\quad \overline{C}_{Q^{Cy}(d,d^2)}(F')=\cone\big(a^{i_j }_j+u_j,\,j=1,\dots, d/2\big).
\end{equation}
Again by Remark \ref{re-optcones}, $\overline{C}_{D^{Cy}(d,d^2)}(F)$ and $\overline{C}_{Q^{Cy}(d,d^2)}(F')$ are pointed cones.   This shows that each point from $\left(\overline{C}_{D^{Cy}(d,d^2)}(F)\cap \overline{C}_{Q^{Cy}(d,d^2)}(F')\right)\setminus\{0\}$ corresponds to a solution to the system
\begin{equation}\label{eq-intC}
\sum_{j=1}^{d/2}a^{i_j }_j\mu_j=\sum_{j=1}^{d/2}(a^{i_j }_j+u_j)\nu_j,\, \mu_j\ge 0,\,\nu_j\ge 0,\,j=1\dots, d/2
\end{equation}
where the $\mu_j$'s are not all equal to $0$ and the $\nu_j$'s are not all equal to $0$.

Since $u_1,\dots,u_{d/2}$ belong to  $V$  and $V\cap \aff\left(\overline{C}_P(F)\right)=\{0\}$ by Lemma \ref{le-subspace}, every solution to the system $\sum_{j=1}^{d/2}a^{i_j }_j\mu_j=\sum_{j=1}^{d/2}(a^{i_j }_j+u_j)\nu_j$ must satisfy
\[\sum_{j=1}^{d/2}u_j\nu_j=0.\]
Since by construction $\conv(u_1,\dots, u_{d/2})$ is a $(d/2-1)$-simplex and $0\in \relint(\conv(u_1,\dots, u_{d/2}))$, the system
\[\sum_{j=1}^{d/2}u_j\nu_j=0,\,\, \nu_j\ge 0,\,j=1\dots,d/2\]
admits a unique (up to scaling) nonzero solution $\bar\nu_j$, and furthermore $\bar\nu_j>0,\,j=1\dots, d/2$.

Therefore the system \eqref{eq-intC} admits a unique (again, up to scaling) solution $\bar{\mu}, \bar{\nu}$, and this solution satisfies $\bar\mu_j= \bar\nu_j>0,\,j=1\dots,d/2$.

By Remark \ref{re-optcones}, we have that
\begin{equation}\label{eq-conerelations2}
 C_{D^{Cy}(d,d^2)}(F)=\relint\left(\overline{C}_{D^{Cy}(d,d^2)}(F)\right)\; \mbox{ and }\; C_{Q^{Cy}(d,d^2)}(F')=\relint\left(\overline{C}_{Q^{Cy}(d,d^2)}(F')\right).
\end{equation}
Let $r:=\sum_{j=1}^{d/2}a^{i_j}_j\bar\mu_j$.
Since $ \bar\mu_j>0$ for $j=1\dots,d/2$, we have that $r\in \relint\left(\overline{C}_{D^{Cy}(d,d^2)}(F)\right)\cap \relint\left(\overline{C}_{Q^{Cy}(d,d^2)}(F')\right)$. Then by \eqref{eq-conerelations2} we have that $C_{D^{Cy}(d,d^2)}(F)\cap C_{Q^{Cy}(d,d^2)}(F')=\{\lambda r, \,\lambda > 0\}$.
\end{proof}

\section{A polyhedral construction}\label{sec:construction}

Let $P_0,\,P_1 \subseteq \R^{d-1}$ be $(d-1)$-polytopes. The  \emph{Cayley embedding} \cite{HRS} of $P_0$ and $P_1$ is the $d$-polytope
$$\conv\left(\begin{pmatrix}P_0 \\ 0 \end{pmatrix}\cup\begin{pmatrix}P_1 \\ 1 \end{pmatrix}\right)\subseteq\R^d,$$
where for a set $S\subseteq \R^{d-1}$ we define notation
$\begin{pmatrix}S \\ \alpha \end{pmatrix}:=\left\{\begin{pmatrix}x \\ \alpha \end{pmatrix}: x\in S\right\}$.

Note that given $x\in\R^{d-1}$, the point $(x,1/2)$ belongs to the Cayley embedding of $P_0$ and $P_1$ if and only if $x\in\frac12P_0+\frac12P_1$.
Some extremal properties of the facial structure of the Minkowski sum of polytopes have been investigated, e.g., in \cite{Sanyal,Weibel}.
However, to the best of our knowledge the construction below is new.

\begin{remark}\label{re-pointinCay}
Let $P_0,P_1 \subseteq \R^{d-1}$ be $(d-1)$-polytopes and $P$ be the Cayley embedding of $P_0$ and $P_1$. Given $F_0\in \mathcal{F}(P_0)$ and $F_1\in \mathcal{F}(P_1)$, let  $F$  be the Cayley embedding of $F_0$ and $F_1$. Then,  given $x\in \R^d$, we have that $x\in F$ if and only if $0\le x_d\le 1$ and there exist $x^0\in \begin{pmatrix}F_0 \\ 0 \end{pmatrix}$ and $x^1\in \begin{pmatrix}F_1 \\ 1 \end{pmatrix}$ such that  $x=(1-x_d)x^0+x_dx^1$ (where $x_d$ is the last component of $x$).
\end{remark}

\begin{lemma}\label{lem:FacesCayley}
Let $P_0,P_1 \subseteq \R^{d-1}$ be $(d-1)$-polytopes and $P$ be the Cayley embedding of $P_0$ and $P_1$. Given $F_0\in \mathcal{F}(P_0)$ and $F_1\in \mathcal{F}(P_1)$, let  $F$  be the Cayley embedding of $F_0$ and $F_1$. Then $F$ is a face of $P$ if and only if $C_{P_0}(F_0)\cap C_{P_1}(F_1)\ne\emptyset$. Furthermore, in this case, given $(r,\alpha)\in\R^{d-1}\times\R$ we have that $(r,\alpha)\in C_P(F)$ if and only if $r\in C_{P_0}(F_0)\cap C_{P_1}(F_1)$ and $\alpha=
\max\{rx:x\in P_0\}-\max\{rx:x\in P_1\}$.
\end{lemma}

\begin{proof} By Remark \ref{re-pointinCay} we have that given $x\in \R^d$,  $x\in P$ if and only if $0\le x_d\le 1$ and there exist $x^0\in \begin{pmatrix}P_0 \\ 0 \end{pmatrix}$ and $x^1\in \begin{pmatrix}P_1 \\ 1 \end{pmatrix}$ such that  $x=(1-x_d)x^0+x_dx^1$. Therefore, given $(r,\gamma)\in \R^{d-1}\times \R$ and $x\in P$, we have that
\begin{equation}\label{eq-pointinCay}
  (r,\gamma)x=(r,\gamma)((1-x_d)x^0+x_dx^1)=(1-x_d)(r,0)x^0+ x_d((r,0)x_1+\gamma)\le (1-x_d)\alpha_0+ x_d(\alpha_1+\gamma),
\end{equation}
where  $\beta_0:=\max\{rx:\,x\in P_0\}$, $\beta_1:=\max\{rx:\,x \in P_1\}$.

Assume now $r\in C_{P_0}(F_0)\cap C_{P_1}(F_1)$, i.e., $rx\le \beta_0$ defines $F_0$ and $rx\le \beta_1$ defines $F_1$. Then if we let  $\gamma=\alpha=\beta_0- \beta_1$, we have that by \eqref{eq-pointinCay} and Remark \ref{re-pointinCay},  the inequality $(r,\gamma)x\le \beta_0$ is valid for $P$ and is satisfied at equality if and only if $x\in F$. Therefore  $F$ is a face of $P$ and  $(r,\alpha)\in C_P(F)$. This proves the ``if'' direction of both equivalences in the statement.
\smallskip

Assume now that $F$ is a face of $P$. Take $(r,\alpha)\in C_P(F)$ and let $\beta$ be such that $(r,\alpha)x\le \beta$ defines $F$.
Then
$$\beta\ge \max\{rx:x\in P_0\}\mbox{ and }\beta-\alpha\ge \max\{rx:x\in P_1\}.$$
Furthermore since  $(r,\alpha)\in C_P(F)$, $F$ is the Cayley embedding of $F_0$, $F_1$, and $F_0$, $F_1$ are both nonempty, the above two inequalities are satisfied at equality. This shows $\alpha=
\max\{rx:x\in P_0\}-\max\{rx:x\in P_1\}$.

We finally show $r\in C_{P_0}(F_0)\cap C_{P_1}(F_1)$. Let $F^*_0$,   $F^*_1$  be the faces of $P_0$, $P_1$ such that $r\in C_{P_0}(F^*_0)\cap C_{P_1}(F^*_1)$ (the existence of $F^*_0$,   $F^*_1$ is guaranteed by 1.\ of Remark \ref{re-optcones}). Assume $F_0\ne F^*_0$ or $F_1\ne F^*_1$, and let $F^*$ be the Cayley embedding of  $F^*_0$,   $F^*_1$. Then $F^*\ne F$ and by the ``if'' part of the lemma, $(r,\alpha)\in C_P(F^*)$. Therefore  $(r,\alpha)\in C_P(F)\cap C_P(F^*)$, a contradiction to  1.\ of Remark \ref{re-optcones}, and this concludes the proof of ``only if'' part.
\end{proof}

Now Remark \ref{re-optcones} and Lemma \ref{lem:FacesCayley} imply the following:

\begin{cor}\label{cor:FacetsCayley}
Let $P_0,P_1 \subseteq \R^{d-1}$ be $(d-1)$-polytopes and $P$ be the Cayley embedding of $P_0$ and $P_1$. Given $F_0\in \mathcal{F}(P_0)$ and $F_1\in \mathcal{F}(P_1)$, let  $F$  be the Cayley embedding of $F_0$ and $F_1$. Then $F$ is a facet of $P$ if and only if $C_{ P_0}(F_0)\cap C_{P_1}(F_1)=\{\lambda r,\, \lambda > 0\}$ for some $r\in \R^d\setminus \{0\}$.
\end{cor}

We now can provide a constructive proof of the following:

\begin{theorem}\label{th-construction}
  For  every even $d \ge 2$ there exists a $(d+1)$-polytope having  $d^{\Omega(d)}$ facets which is the Cayley embedding of $d$-polytopes $P_1$ and $P_2$, each having $ d^2$ facets. \end{theorem}

  \begin{proof}  Let $d\ge 2$ be even and fix a coloring of the facets of $D^{Cy}(d,d^2)$.
  Let $F$ be a   colorful  face of  $D^{Cy}(d,d^2)$ and $F'$ be the corresponding face of   $Q^{Cy}(d,d^2)$. By Lemma \ref{le-intoptcones} we have that
$C_{ D^{Cy}(d,d^2)}(F)\cap C_{ Q^{Cy}(d,d^2)}(F')=\{\lambda r,\, \lambda > 0\}$ for some $r\in \R^d\setminus \{0\}$. By Corollary \ref{cor:FacetsCayley}, the Cayley embedding of $F$, $F'$ is a facet of the Cayley embedding of  $D^{Cy}(d,d^2)$ and $Q^{Cy}(d,d^2)$.

By Remark \ref{re-polarcyclicpoly},  the intersection of every $d/2$ facets of $D^{Cy}(d,d^2)$ forms a distinct face. By  definition, the number of colorful $(d/2)$-faces of $D^{Cy}(d,d^2)$ is $(2d)^{d/2}=d^{\Omega(d)}$.  Therefore the Cayley embedding of  $P_1:=D^{Cy}(d,d^2)$ and $P_2:=Q^{Cy}(d,d^2)$ has $d^{\Omega(d)}$ facets. Since $P_1$ and $P_2$ have $d^2$ facets each, this proves the theorem.
\end{proof}

 As shown in Section \ref{sec:strategy}, the above theorem implies Theorem \ref{thr:main}.

\section{Proof of Theorem \ref{thr:main-approx}}\label{sec:Proof-of-approx}

The $d$-dimensional \emph{cross-polytope} is (see e.g.~\cite{Ziegler}):
$$Q_d^\triangle:=\Bigg\{x \in \R^d : \sum_{i \in I} x_i - \sum_{i \in [d]\setminus I}x_i \leq 1 :\ \forall\, I \subseteq [d]\Bigg\},$$
with $[d]:=\{1,\dots,d\}$. $Q_d^\triangle$ has $2^d$ facets, as  every inequality in the above description defines a facet.

Consider the following  $(d-1)$-simplices:
$$ P_1:=\left\{x\in [0,1]^d:\,\sum_{i=1}^{d}x_i=1\right\},\;\;\;P_{-1}:=\left\{x\in [-1,0]^d:\,\sum_{i=1}^{d}x_i=-1\right\}.$$
Since $Q_d^\triangle$ has $2d$ vertices, namely  $\pm e_i$ for $i=1,\dots,d$ (the unit vectors and their negatives), it follows that $Q_d^\triangle=\conv(P_1\cup P_{-1})$. Therefore  $Q_d^\triangle$ is a $d$-polytope with $2^d$ facets which is the convex hull of two of its facets that are $(d-1)$-simplices. (This  choice is not unique: any two parallel facets of $Q_d^\triangle$ will do.)

We show that, for every constant $\varepsilon>0$, any $\varepsilon$-approximation of the cross-polytope must still have an exponential number of facets. We will then invoke Lemma \ref{lem:hammer} to conclude the proof of Theorem \ref{thr:main-approx}.

The following observation allows us to focus on $\varepsilon$-approximations that only use facet-defining inequalities.

\begin{lemma}\label{lem:cara}
	Let $Q\subseteq \R^d$ be an $\varepsilon$-approximation of a $d$-polytope $P\subseteq \R^d$ for some $\varepsilon >0$. Then there exists an $\varepsilon$-approximation of $P$ with at most $d\cdot f(Q)$ inequalities, each of which defines a facet of $P$.
\end{lemma}

\begin{proof}
	Let $cx \leq \delta$ be any facet-defining inequality for $Q$. Then $cx \leq \delta$ is valid for $P$ and without loss of generality we may assume that $cx \leq \delta$ is supporting for $P$. Hence, by Caratheodory's theorem, it is a conic combination of at most $d$ facet-defining inequalities for $P$. Hence, we can replace $cx \leq \delta$ with the facet-defining inequalities for $P$ that define it and obtain a polytope $P'$ such that $P\supseteq P'\supseteq Q$. By repeating the procedure for all facet-defining inequalities for $Q$, we obtain the claimed result.
\end{proof}

\begin{lemma}\label{lem:approx-cross}
Given $\varepsilon>0$, there exists $\kappa>1$ such that every $\varepsilon$-approximation of $Q_d^\triangle$ has $\Omega(\kappa^d)$ facets.
\end{lemma}

\begin{proof}We exhibit a set ${\cal S}$ of points that \emph{cannot} belong to any $\varepsilon$-approximation of $Q_d^\triangle$, but such that any facet-defining inequality for $Q_d^\triangle$ can cut off \emph{at most} $t$ of them. Hence by Lemma \ref{lem:cara}, the number of inequalities needed to describe an $\varepsilon$-approximation of $Q_d^\triangle$ is at least   $|{\cal S}|/(dt)$. (Our proof approach can be interpreted as an extension of those in \cite{FS,KW}.)
	
	Let $\delta>2\varepsilon$ be fixed. Consider the family ${\cal S}\subseteq \R^d$ of $2^d$ points having coordinates equal to $\pm(1+\delta)/d$.
	
	\begin{claim}
		Let $x^* \in {\cal S}$. Then $x^*$ does not belong to any $\varepsilon$-approximation of $P$.
	\end{claim}
	
	\begin{cpf}
		Let $c$ be the objective function with $c_i=1$ if $x^*_i>0$ and $c_i=-1$ if $x^*_i<0$. Then for any polytope $P'$ that contains $Q_d^\triangle\cup \{x^*\}$ we have:
		$$\underset{x \in P'}{\max}\, cx - \underset{x \in P'}{\min}\, cx \geq (1+\delta) - (-1) = 2 + \delta > 2(1+\varepsilon)= (1+\varepsilon)\big(\underset{x \in Q_d^\triangle}{\max}\, cx - \underset{x \in Q_d^\triangle}{\min}\, cx\big).$$
		Therefore $P'$ is not an $\varepsilon$-approximation to $Q_d^\triangle$.
	\end{cpf}
	
	\begin{claim}
		Let $I\subseteq [d]$. There exists $\bar\kappa<2$ such that $\sum_{i \in I} x_i - \sum_{i \in [d]\setminus I}x_i \leq 1$ is violated by at most $\bar\kappa^d$ points from ${\cal S}$. \end{claim}
	
	\begin{cpf}
		By the symmetry of $Q_d^\triangle$, it suffices to prove the statement for the inequality $\sum_{i \in [d]} x_i \leq 1$. Fix $x^* \in {\cal S}$ and suppose $t$ of its components are positive. Then
		$$\sum_{i \in [d]} x^*_i = t \frac{1+\delta}{d}  - (d-t)\frac{1+\delta}{d}  = 2t \frac{1+\delta}{d} - 1- \delta,$$ hence the inequality is violated if and only if $$2t \frac{1+\delta}{d} - 1- \delta> 1 \quad \Leftrightarrow \quad t > \frac{d}{2}\cdot \frac{2+\delta}{1+\delta}.$$
		Define $\gamma:=\frac{1}{2}\cdot\frac{2+\delta}{1+\delta}>\frac{1}{2}$. Then a point in ${\cal S}$ violates $\sum_{i \in [d]}x_i \leq 1$ if and only if it has more than $\gamma d$ positive entries. The number of points with this property is upper bounded by
		$$\sum_{j=\lceil\gamma d\rceil}^{d} {d \choose j} =\sum_{j=0}^{\lfloor(1-\gamma)d\rfloor} {d \choose j}\leq 2^{dH(1-\gamma)},$$ where we used the well-known bound $\sum_{j=0}^k {n \choose j}\leq 2^{nH(k/n)}$ that is valid for $k\leq \frac{n}{2}$ and uses the entropy function $H(p)=-x \log_2(p) - (1-p) \log_2 (1-p)$ (see e.g. \cite{Jukna}). It is well-known that $H(p)<H(1/2)=1$ for all $0\leq p<1/2$. Since $\varepsilon$ and $\delta$ are fixed, $1-\gamma$ is a constant strictly less than $1/2$, and thus we conclude that $2^{dH(1-\gamma)}\leq \bar\kappa^d$ for some $\bar\kappa< 2$.
	\end{cpf}
	
	Putting everything together, the number of inequalities needed to describe an $\varepsilon$-approximation of $Q_d^\triangle$ is at least
	$$\frac{|{\cal S}|}{d\bar \kappa^d}=\frac{1}{d}\left(\frac{2}{\bar\kappa}\right)^d=\Omega(\kappa^d) $$
	for some $\kappa > 1$, as required.
	\end{proof}

\begin{proof}[Proof of Theorem \ref{thr:main-approx}]
	Fix $\varepsilon>0$, and let $P'$ be an $\varepsilon$-approximation of $Q_d^\triangle$. Then $P'$ has $\Omega(\kappa^d)$ facets for some $\kappa>1$ by Lemma \ref{lem:approx-cross}. Recall that $Q_d^\triangle$ is the convex hull of two polytopes with $d+1$ facets each. By Lemma \ref{lem:hammer}, every linear extension of $P'$ with a number of facets polynomial in $d$ has $$\Omega\left(\frac{\log \kappa^d}{\log d^t}\right)=\Omega(d/\log d)$$ additional variables, as required.
\end{proof}

\section{A consequence for lift-and-project}\label{sec:split}

Given $P\subseteq [0,1]^d$, the lift-and-project method of Balas, Ceria and Cornu\'ejols \cite{BaCeCo} iteratively constructs polyhedral relaxations of $P\cap \Z^n$ that are the convex hull of the two faces of $P$ defined by $x_j\ge 0$ and $x_j\le 1$, for some $j=1,\dots,d$.

We show that even in this restrictive setting,  Theorem \ref{thr:main} is the best possible. More precisely, we prove the following:

\begin{theorem}\label{thr:our-split}  Fix a polynomial $\sigma$. For each odd $d \in \N$, there exists a formulation of   a nonempty polytope $P\subseteq [0,1]^d $ of size $f$, such that  any formulation of $$\conv((P\cap\{x\in \R^d:x_d=0\})\cup (P\cap\{x\in \R^d:x_d=1\}))$$ of size at most  $\sigma(f)$ has $\Omega(d)$ additional variables.
\end{theorem}

Given polytope $Q\subseteq \R^d$ and $x^* \not\in \aff(Q)$, let $C$ be the polyhedral cone generated by
the vectors $\{x - x^* : x\in Q\}$.  The polyhedron $\hom(Q,x^*):=x^*+C$ is the \emph{homogenization} of $Q$ with respect to $x^*$. Note that $F$ is a facet of $Q$ if and only if  $\hom(F,x^*)$ is a facet of  $\hom(Q,x^*)$ and all facets of  $\hom(Q,x^*)$ arise in this way.

\begin{remark}\label{re-hom}   Let $P_0,P_1\in \R^{d-1}$ be $(d-1)$-polytopes and pick $x^0$, $x^1$ in the interior of $P_0$, $P_1$ respectively. There exists $\varepsilon>0$ such that
$$ H_0:=\hom\left(\begin{pmatrix}P_0\\0\end{pmatrix},\begin{pmatrix}x^0\\-\varepsilon\end{pmatrix}\right)\mbox{ contains }\begin{pmatrix}P_1\\1\end{pmatrix}\mbox{ in its interior, and }$$
$$ H_1:=\hom\left(\begin{pmatrix}P_1\\1\end{pmatrix},\begin{pmatrix}x^1\\1+\varepsilon\end{pmatrix}\right)\mbox{ contains }\begin{pmatrix}P_0\\0\end{pmatrix}\mbox{ in its interior.}$$
In particular, if $\bar x$ is a vertex of $ H_0\cap H_1$, then
  $\bar x=\begin{pmatrix}x^0\\-\varepsilon\end{pmatrix}$ or  $\bar x=\begin{pmatrix}x^1\\1+\varepsilon\end{pmatrix}$ or $0<\bar{x}_d<1$.
  \end{remark}

Given $P_0:=D^{Cy}(d-1,(d-1)^2)$ and  $P_1:=Q^{Cy}(d-1,(d-1)^2) \subseteq \R^{d-1}$, let $\varepsilon>0$, $H_0$ and $H_1$ be as in Remark \ref{re-hom}.  By possibly scaling the first $d-1$ coordinates, we may assume that the polytope $P:=H_0\cap H_1\cap \{x\in \R^d:0\le x_d\le 1\}$ is contained in $[0,1]^d$. Note that  $\begin{pmatrix}P_0\\0\end{pmatrix}$ is the facet of $P$ defined by the inequality $x_d\ge 0$ and $\begin{pmatrix}P_1\\1\end{pmatrix}$ is the facet of $P$ defined by the inequality $x_d\le 1$.

\begin{proof}[Proof of Theorem \ref{thr:our-split}.] Let $P$ be the polytope defined as above. By Remark \ref{re-hom}, $P$ has $2(d-1)^2+2$ facets. The proof of  Theorem \ref{th-construction} shows that  the polytope
$$\conv((P\cap\{x\in \R^d:x_d=0\})\cup (P\cap\{x\in \R^d:x_d=1\}))$$
has $d^{\Omega(d)}$ facets. By Lemma \ref{lem:hammer}, any formulation of the above polytope has $\Omega(d)$ additional variables.
  \end{proof}

\paragraph{Acknowledgments} We are grateful to Volker Kaibel and Stefan Weltge for helpful discussions and suggestions regarding the present work.

\end{document}